\documentclass[a4paper]{amsart}

\usepackage[english]{babel}
\usepackage[utf8]{inputenc}

\usepackage{amssymb}
\usepackage{amsmath}
\usepackage{amsthm}
\usepackage{bbm}
\usepackage[normalem]{ulem}

\usepackage{enumerate}
\usepackage{tikz}
\usepackage{marginnote}

\newcommand{\bbC}{\mathbb{C}}

\newcommand{\bbN}{\mathbb{N}}

\newcommand{\bbR}{\mathbb{R}}

\newcommand{\bbZ}{\mathbb{Z}}

\newcommand{\calA}{\mathcal{A}}

\DeclareMathOperator{\Real}{Re}

\DeclareMathOperator{\diver}{div}

\newcommand{\R}{\mathbb{R}}

\newcommand{\mG}{\mathsf{G}}
\newcommand{\mD}{\mathsf{D}}
\newcommand{\mV}{\mathsf{V}}
\newcommand{\mE}{\mathsf{E}}
\newcommand{\mv}{\mathsf{v}}
\newcommand{\mw}{\mathsf{w}}
\newcommand{\me}{\mathsf{e}}

\DeclareMathOperator{\dist}{dist}

\DeclareMathOperator{\id}{id}
\DeclareMathOperator{\one}{\mathbbm{1}}
\DeclareMathOperator{\re}{Re}

\newcommand{\argument}{\mathord{\,\cdot\,}}
\newcommand{\dx}{\;\mathrm{d}}
\DeclareMathOperator{\sgn}{sgn}

\newcommand{\spec}{\sigma}
\newcommand{\Res}{\mathcal{R}}
\newcommand{\per}{{\operatorname{per}}}

\DeclareMathOperator{\spb}{s}
\newcommand{\norm}[1]{\left\lVert #1 \right\rVert}
\newcommand{\modulus}[1]{\left\lvert #1 \right\rvert}

\theoremstyle{definition}
\newtheorem{definition}{Definition}[section]
\newtheorem{remark}[definition]{Remark}

\newtheorem{example}[definition]{Example}

\theoremstyle{plain}
\newtheorem{prop}[definition]{Proposition}

\newtheorem{theorem}[definition]{Theorem}

\numberwithin{equation}{section}

\begin{document}

\title[Eventual Domination]{Eventual Domination for Linear Evolution Equations}

\author{Jochen Gl\"uck}
\address{Jochen Gl\"uck, Faculty of Computer Science and Mathematics, University of Passau, Innstra{\ss}e 33, D-94032 Passau, Germany}
\email{jochen.glueck@uni-passau.de}

\author{Delio Mugnolo}
\address{Delio Mugnolo, Faculty of Mathematics and Computer Science, University of Hagen, D-58084 Hagen, Germany}
\email{delio.mugnolo@fernuni-hagen.de}

\date{\today}

\dedicatory{Dedicated with great pleasure to our teacher and friend Wolfgang Arendt on the occassion of his 70th birthday}
\thanks{The work of Delio Mugnolo was partially supported by the Deutsche Forschungsgemeinschaft (Grant 397230547).}

\begin{abstract}
	We consider two $C_0$-semigroups on function spaces or, more generally, Banach lattices and give necessary and sufficient conditions for the orbits of the first semigroup to dominate the orbits of the second semigroup for large times. As an important special case we consider an $L^2$-space and self-adjoint operators $A$ and $B$ which generate $C_0$-semigroups; in this situation we give criteria for the existence of a time $t_1 \ge 0$ such that $e^{tB} \ge e^{tA}$ for all subsequent times $t\ge t_1$. 
	
As a consequence of our abstract theory, we obtain many surprising insights into the behaviour of various second and fourth order differential operators.
\end{abstract}

\maketitle

\section{Introduction} \label{section:introduction}

Ever since Simon's pioneering investigation in~\cite{Sim77} it is a classical question in the theory of $C_0$-semigroups whether a given semigroup $(e^{tB})_{t \ge 0}$ \emph{dominates} a semigroup $(e^{tA})_{t \ge 0}$, meaning that $\lvert e^{tA}f\rvert \le e^{tB}\lvert f\rvert$ for all times $t$ and all $f$ in the underlying Banach space $E$. Of course, this only makes sense if $E$ is endowed with an appropriate order structure, i.e., if $E$ is for instance an $L^p$-space or, more generally, a Banach lattice.

It is known since~\cite{HesSchUhl77} that domination of semigroups can be characterised in terms of their generators; we refer to \cite[Section~C-II-4]{Nag86} for an extension to the case of semigroups acting on general Banach lattices. For the important special case of semigroups generated by operators on $L^2$ associated with quadratic forms, domination can also be described in terms of the forms and their domains. This is a classical result of Ouhabaz \cite[Section~3]{Ouh96}, which has recently been generalised into a different direction \cite{LenSchWir20}.

It should not come as a surprise that domination is closely related to \emph{positive} semigroups. On the other hand, the behaviour of several concrete evolution equations -- namely the bi-harmonic heat equation on $\mathbb{R}^n$ \cite{FerGazGru08, Ferrero2008} and the Dirichlet-to-Neumann semigroup on the two-dimensional unit circle \cite{Daners2014} -- have recently given rise to a generalisation of the concept of positive semigroups to so-called \emph{eventually positive} semigroups. Those are, roughly speaking, semigroups $(e^{tA})_{t \ge 0}$ with the property that $e^{tA}f \ge 0$ for all \textit{sufficiently large} times $t$ (as opposed to: for all $t\ge 0$) and all $f \ge 0$. An abstract and general approach to study this behaviour first appeared in the papers \cite{Daners2016, DanGluKen16b}; those papers also contain various further examples of the eventual positivity phenomenon. Later, the theory was further developed in \cite[Part~II]{GlueckDISS} und \cite{Daners2017, DanersPERT, DanersUNIF}. Recent applications to bi- and polyharmonic equations on networks with stationary or dynamic boundary conditions can be found in \cite{GregorioMugnolo,GregorioPreprint}.

The successful set-up of a theory of eventual positivity, together with numerous examples that are now known for this behaviour, suggests that a similar generalisation of semigroup \emph{domination} is a worthwhile endeavour. The present paper is a first contribution towards such a theory of \emph{eventual domination}. To keep things handy, we restrict ourselves to the case where two semigroups $(e^{tA})_{t \ge 0}$ and $(e^{tB})_{t \ge 0}$ are already known to be eventually positive; i.e., again roughly speaking, we assume that $e^{tA}f \ge 0$ and $e^{tB}f \ge 0$ for all $f \ge 0$ and all sufficiently large times, and we intend to characterise whether we also have $e^{tB}f \ge e^{tA}f$ for all sufficiently large times.

Of course, several technical details have to be taken into account when we make this precise. For instance, it has to be clarified whether the time $t_0$ from which on $e^{tB}f \ge e^{tA}f$ holds is allowed to depend on the initial value $f$; this yields two concepts that can be described as \emph{individual} and \emph{uniform} eventual domination. Moreover, we will also replace mere positivity and domination with somewhat stronger concepts which are more interesting in applications and which also have several theoretical advantages; see the beginning of Section~\ref{section:a-characterisation-of-eventual-dommination} for details.

While the present article is, to the best of our knowledge, the first systematic treatment of eventual domination between operator semigroups, we note that related notions and results have occasionally occurred in the literature. For instance, \cite[Corollary~4.8.8]{Davies1989} contains an eventual domination result for certain heat kernels. The property of \emph{asymptotic domination} between operator semigroups -- which is clearly related to eventual domination -- has been studied in \cite{Emelyanov2001}, but the focus of that paper was on convergence of the semigroups as $t \to \infty$ rather than on characterisations of eventual or asymptotic domination.

\subsection*{Organisation of the article}

In Section~\ref{section:a-motivating-example} we start with two concrete examples to give some additional motivation for our analysis. In Section~\ref{section:a-characterisation-of-eventual-dommination} we first discuss several preliminaries and then proceed with the statements and proofs of our main results. In the final Section~\ref{section:applications} we give a variety of examples which illustrate how to apply our main results and which give evidence of the prevalence of the phenomenon \emph{eventual domination} in the realm of concrete evolution equations.

\section{Two motivating examples} 
\label{section:a-motivating-example}

\subsection{The heat equation with different types of boundary conditions}
\label{subsection:heat-equation-different-boundary-conditions}

Let us consider the one-dimensional heat equation 
\[
u_t(t,x)=u_{xx}(t,x),\qquad t\ge 0,\ x\in (0,1)
\]
on the interval $(0,1)$ and discuss two different types of boundary conditions:
\begin{itemize}
	\item Mixed Dirichlet and Neumann: $u(t,0)=0$ and $u_x(t,1)=0$ for all $t\ge 0$;
	\item Periodic: $u(t,0)=u(t,1)$ and $u_x(t,0)=u_x(t,1)$ for all $t\ge 0$.
\end{itemize}
The corresponding Cauchy problems on $L^2(0,1)$ are governed by two semigroups: we will denote them by $(e^{t\Delta^M})_{t\ge 0}$ and $(e^{t\Delta^P})_{t\ge 0}$, respectively. Both are strongly continuous, contractive, and positive. We are interested in the following question: does one of these semigroups dominate the other one? 

A handy characterization of domination for semigroups associated with closed forms has been discovered by Ouhabaz in the 1990s: let us recall the following interesting special case~\cite[Theorem~2.21]{Ouh05}.

\begin{prop}\label{prop:ouh-crit}
	Let $(\Omega,\mu)$ be a $\sigma$-finite measure space and let $a,b$ be accretive, continuous, and closed sesquilinear forms with dense domains $D(a),D(b)$ on $L^2 (\Omega, \mu)$. We denote by $(e^{tA} )_{t\ge 0}$, $(e^{tB} )_{t\ge 0}$ their associated semigroups, respectively. Assume $(e^{tB} )_{t\ge 0}$ to be positive. Then $(e^{tA} )_{t\ge 0}$ is dominated by $(e^{tB} )_{t\ge 0}$ if and only if the following two assertions are satisfied:
	\begin{enumerate}[\upshape (a)]
		\item $D(a)$ is an \emph{ideal} of $D(b)$, meaning that
			\begin{itemize}
				\item $u\in D(a)$ implies $|u|\in D(b)$ and
				\item $v\sgn u\in D(a)$ for all $u\in D(a)$ and $v\in D(b)$ that satisfy $|v|\le |u|$;			
			\end{itemize}
		and additionally
		\item $\Real a(u,v)\ge b(|u|,|v|)$ for all $u,v\in D(a)$ in such that $u\overline{v} \ge 0$.
		\end{enumerate}
\end{prop}

Both generators $\Delta^M$ and $\Delta^P$ are associated with a realization of the quadratic form
\[
	(u,v)\mapsto\int_0^1 u_x(x)\overline{v_x(x)}\dx x,
\]
with form domains
\[
	\{u \in H^1(0,1): \; u(0) = 0\} \quad \text{and} \quad \{u\in H^1(0,1):u(0)=u(1)\},
\]
respectively. Neither of these two form domains is an ideal in the other one, so neither does $(e^{t\Delta^M})_{t\ge 0}$  dominate $(e^{t\Delta^P})_{t\ge 0}$, nor vice versa. 

However, it is not difficult to see that we have a certain type of \emph{eventual domination}: consider a non-zero function  $0 \le f \in L^2(0,1)$. Then $e^{t\Delta^P}f$ converges to $\langle \one, f \rangle \one$ as $t \to \infty$. On the other hand, $\norm{e^{t\Delta^M}f}_\infty \to 0$ as $t \to \infty$, so it follows that $e^{t \Delta^P}f \ge e^{t \Delta^M}f$ for all $t$ from a certain time $t_0$ on.

In this article we pursue two goals:
\begin{enumerate}[(1)]
	\item We formalize the above observations on an abstract level to characterize under which conditions a semigroup $(e^{tB})_{t \ge 0}$ eventually dominates a semigroup $(e^{tA})_{t \ge 0}$;
	
	\item We give sufficient conditions for this eventual domination to be \emph{uniform} in the sense that the time $t_0$ from which on we have $e^{tB}f \ge e^{tA}f$ can be chosen to be independent of $f$.
	
	This uniform eventual domination is indeed satisfied in the above example, as a consequence of Theorem~\ref{thm:main-result-uniform} below (see Subsection~\ref{subsec:1-D-laplace-mixed-periodic} for details).
\end{enumerate}

\subsection{Comparing the heat semigroup at different times} \label{subsection:heat-semigroup-at-different-times}

Let $\Omega \subseteq \bbR^d$ be a bounded domain; here, and throughout the article, we use the notion \emph{domain} shorthand for a non-empty, open and connected set in $\bbR^d$. Assume that $\Omega$ has Lipschitz boundary and let $\Delta^D$ denote the realization of the Dirichlet Laplacian on $L^2(\Omega)$. Then $e^{t\Delta^D}$ converges to $0$ with respect to the operator norm as $t \to \infty$. Moreover, the semigroup generated by twice the Laplacian, $(e^{2t\Delta^D})_{t \in [0,\infty)}$, converges faster to $0$ in the sense that
\begin{align*}
	\frac{\norm{e^{2 t\Delta^D}}}{\norm{e^{t\Delta^D}}} \to 0 \qquad \text{as } t \to \infty.
\end{align*}
Given that we also have $\norm{e^{2 t\Delta^D}} \le \norm{e^{t\Delta^D}}$ for each $t \in [0,\infty)$, it is thus natural to wonder whether a domination property of the type
\begin{align}
	\label{eq:domination-heat-semigroup}
	e^{2 t\Delta^D} \le e^{t\Delta^D}
\end{align}
holds. It follows from a simple and very general result on positive semigroups (see Proposition~\ref{prop:non-decreasing-orbits} below) that~\eqref{eq:domination-heat-semigroup} can in fact \emph{not} hold for all times $t \in [0,\infty)$ since this would imply that each operator $e^{t\Delta^D}$ is a multiplication operator. However, it follows from one of our main results, Theorem~\ref{thm:main-result-uniform}, that~\eqref{eq:domination-heat-semigroup} does indeed hold for all sufficiently large times $t$. In other words, the semigroup $(e^{t\Delta^D})_{t \in [0,\infty)}$ eventually dominates the semigroup $(e^{2t\Delta^D})_{t \in [0,\infty)}$. This will be discussed in more detail -- and in a more general setting -- in Subsection~\ref{subsection:non-monotonicity-of-semigroups}; see in particular Example~\ref{ex:eventual-monotonicity-for-dirichlet-laplacian}.

\section{A characterisation of eventual domination} \label{section:a-characterisation-of-eventual-dommination}

\subsection{Setting the stage}

In this subsection, we briefly discuss a few notions from Banach lattice and operator theory which we need throughout the article.

\subsubsection*{Banach lattices} Throughout we assume that the reader is familiar with the theory of Banach lattices; classical references for this theory are for instance \cite{Schaefer1974} and \cite{Meyer-Nieberg1991}. As a guiding principle one might, in the present article, always think of an $L^p$-space over a $\sigma$-finite measure space; this example class is particularly important in Section~\ref{section:applications} where we apply our results to various differential operators. 

Let $E$ be a complex Banach lattice with positive cone $E_+$. We call the elements of $E_+$ the \emph{positive} vectors in $E$. For two vectors $f,g \in E$ we write $f \le g$ iff $f$ and $g$ are both contained in the real part of $E$ and $g-f \in E_+$; moreover, we use the notation $f < g$ shorthand for $f \le g$ and $f \not= g$. Given a vector $0 < u \in E$ and another vector $f\in E$, we write $f \gg_u 0$ if there exists a number $\varepsilon > 0$ such that $f \ge \varepsilon u$.

The concept of \emph{principal ideals} and of \emph{quasi-interior points} is of particular importance for us. For each $u \in E_+$ we call the set
\begin{align*}
	E_u := \{f \in E: \; \exists c \in [0,\infty) \text{ such that } \modulus{f} \le cu \}
\end{align*}
the \emph{principal ideal} generated by $u$ in $E$; it is indeed an ideal in $E$ in the sense of vector lattices. The ideal $E_u$ becomes itself a complex Banach lattice when endowed with the \emph{gauge norm $\norm{\argument}_u$} given by
\begin{align*}
	\norm{f}_u = \inf \{c \in [0,\infty): \; \modulus{f} \le cu\}.
\end{align*}
The gauge norm is always at least as strong as the norm inherited from $E$. It is instructive to keep the following example in mind: if $E = L^p(\Omega,\mu)$ for a finite measure space $(\Omega,\mu)$ and for $p \in [1,\infty]$, and if $u = \one \in L^p(\Omega,\mu)$, then the principal ideal $E_u = (L^p(\Omega,\mu))_{\one}$ is exactly the space $L^\infty(\Omega,\mu)$, and the gauge norm $\norm{\argument}_u$ is simply the sup norm. In particular, we see in this example that a principal ideal in a Banach lattice might or might no be closed and that the gauge norm might be strictly stronger than or equal to the norm inherited from $E$.

Let again $u \in E_+$. If the principal ideal $E_u$ is dense in $E$, then the vector $u$ is called a \emph{quasi-interior point} of the positive cone $E_+$. For $E = L^p(\Omega,\mu)$, where $(\Omega,\mu)$ is a $\sigma$-finite measure space and where $p \in [1,\infty)$, a vector $0 \le u \in L^p(\Omega,\mu)$ is a quasi-interior point of the positive cone if and only if $u(\omega) > 0$ for almost all $\omega \in \Omega$. On the other hand, a vector $0 \le u \in L^\infty(\Omega,\mu)$ is a quasi-interior point of the positive cone of $L^\infty(\Omega,\mu)$ if and only if there exists a number $\delta > 0$ such that $u \ge \delta \one$.

We call a bounded linear operator $T$ on $E$ \emph{positive}, and denote this by $T \ge 0$, if $TE_+ \subseteq E_+$. For two bounded operators $T,S$ on $E$ we write $T \le S$ if both operators map real elements of $E$ to real elements and if, in addition, $S-T \ge 0$.

\subsubsection*{Operators, spectral theory and $C_0$-semigroups}

Let $E$ be a complex Banach space. We denote the \emph{spectrum} of any operator $A: E \supseteq D(A) \to E$ by $\spec(A)$. For $\lambda \in \bbC \setminus \spec(A)$ we use the notation $\Res(\lambda,A) := (\lambda - A)^{-1}$ for the \emph{resolvent} of $A$ at $\lambda$. The value
\begin{align*}
	\spb(A) := \sup \{\re \lambda: \; \lambda \in \spec(A)\} \in [-\infty,\infty]
\end{align*}
is called the \emph{spectral bound} of $A$, and if $\spb(A) \in \bbR$, then the set
\begin{align*}
	\spec_\per(A) := \spec(A) \cap \big(i\bbR + \spb(A)\big) = \{\lambda \in \spec(A): \; \re \lambda = \spb(A)\}
\end{align*}
is the so-called \emph{peripheral spectrum} of $A$. If an operator $A: E \supseteq D(A) \to E$ generates a $C_0$-semigroup, then we denote this semigroup by $(e^{tA})_{t \ge 0}$.

We denote the dual space of $E$ by $E'$. The dual operator of a densely defined linear operator $A$ on $E$ is denoted by $A'$. For every $f \in E$ and every $\varphi \in E'$ we denote by $f \otimes \varphi$ the bounded linear operator on $E$ given by $(f \otimes \varphi)g = \langle \varphi, g\rangle f$ for all $g \in E$.

\subsubsection*{Real operators}

Let $E$ be a complex Banach lattice. Then $E$ possesses an underlying real Banach lattice $E_\bbR$ which we call the \emph{real part} of $E$. If $E$ is a (complex-valued) $L^p$-space over some measure space, then the real part of $E$ is precisely its subset of real-valued $L^p$-functions.

A linear operator $A: E \supseteq D(A) \to E$ is called \emph{real} if
\begin{align*}
	D(A) = D(A)\cap E_\bbR + i D(A) \cap E_\bbR
\end{align*}
and if $A$ maps $D(A) \cap E_\bbR$ to $E_\bbR$. Consequently, a bounded linear operator $T$ on $E$ is real if and only if it leaves $E_\bbR$ invariant. We call a $C_0$-semigroup $(e^{tA})_{t \ge 0}$ on $E$ \emph{real} if the operator $e^{tA}$ is real for each time $t \ge 0$. It is not difficult to see that a $C_0$-semigroup is real if and only if its generator is real.

\subsubsection*{Eventual positivity}

Here, we recall a few notions from the theory of eventually positive semigroups. Let $E$ be a complex Banach lattice and let $(e^{tA})_{t \ge 0}$ be a $C_0$-semigroup on $E$. There are two relevant notions for us -- mere eventual positivity and eventual strong positivity -- and for each of them, an individual and a uniform version exist. Here are the precise definitions:

\begin{itemize}
	\item \emph{Eventual positivity:} Our semigroup $(e^{tA})_{t \ge 0}$ is called \emph{individually eventually positive} if for each $0 \le f \in E$ there exists a time $t_0 \ge 0$ such that $e^{tA}f \ge 0$ for all $t \ge t_0$.
	
	If, in addition, $t_0$ can be chosen to be independent of $f$, then the semigroup is called \emph{uniformly eventually positive}.
	
	\item \emph{Eventual strong positivity:} Fix a quasi-interior point $u$ of $E_+$. Our semigroup $(e^{tA})_{t \ge 0}$ is called \emph{individually eventually strongly positive with respect to $u$} if for each $0 < f \in E$ there exists a time $t_0 \ge 0$ such that $e^{tA}f \gg_u 0$ for all $t \ge t_0$.
	
	If, in addition, $t_0$ can be chosen to be independent of $f$, then the semigroup is called \emph{uniformly eventually strongly positive with respect to $u$}.
\end{itemize}

For an in-depth study of these notions, and for the discussion of many examples, we refer the reader to the series of papers \cite{Daners2016, DanGluKen16b, Daners2017, DanersPERT, DanersUNIF} and to \cite[Part~III]{GlueckDISS}.

\subsection{Individual eventual domination}

The following theorem is our first main result.

\begin{theorem} \label{thm:main-result-individual}
	Let $E$ be a complex Banach lattice and let $u$ be a quasi-interior point of $E_+$. Consider two distinct real $C_0$-semigroups $(e^{tA})_{t \in [0,\infty)}$ and $(e^{tB})_{t \in [0,\infty)}$ on $E$ which satisfy the following assumptions:
	\begin{itemize}
		\item \emph{Spectral theoretic assumptions:} Both generators $A$ and $B$ have non-empty spectrum and the peripheral spectra $\spec_\per(A)$ and $\spec_\per(B)$ consist of poles of the resolvents $\Res(\argument,A)$ and $\Res(\argument,B)$, respectively;
		\item \emph{Smoothing assumptions:}
		 Both semigroups $(e^{tA})_{t \in [0,\infty)}$ and $(e^{tB})_{t \in [0,\infty)}$ are analytic; moreover, there exists a time $t_0 \in [0,\infty)$ such that $e^{t_0A} E \subseteq E_u$ and $e^{t_0B}E \subseteq E_u$; and
		\item \emph{Positivity assumptions:} The semigroup $(e^{tA})_{t \in [0,\infty)}$ is individually eventually positive, and the semigroup $(e^{tB})_{t \in [0,\infty)}$ is individually eventually strongly positive with respect to $u$.
	\end{itemize}
	Then the following assertions are equivalent:
	\begin{enumerate}[\upshape (i)]
		\item For all $0 < f \in E$ there exists a time $t_1 \in [0,\infty)$ such that $e^{tB}f \ge e^{tA}f \ge 0$ for all $t \ge t_1$, i.e., $(e^{tB})_{t \in [0,\infty)}$ \emph{individually eventually dominates} $(e^{tA})_{t \in [0,\infty)}$.
		\item For all $0 < f \in E$ there exists a time $t_1 \in [0,\infty)$ such that $e^{tB}f \gg_u e^{tA}f \ge 0$ for all $t \ge t_1$.
		\item We have $\spb(B) > \spb(A)$.
	\end{enumerate}
\end{theorem}

Assertions~(i) and~(ii) in the above theorem both state in a way that the semigroup $(e^{tB})_{t \in [0,\infty)}$ eventually dominates the semigroup $(e^{tA})_{t \in [0,\infty)}$; we introduced the terminology ``individually eventually dominates'' in the theorem since the time $t_1$ from which on this happens can depend on the initial value $f$.

\begin{remark} \label{rem:necessary-condition-for-domination}
	If we consider two \emph{positive semigroups} $(e^{tA})_{t \in [0,\infty)}$ and $(e^{tB})_{t \in [0,\infty)}$, Theorem~\ref{thm:main-result-individual} shows in particular that, under the various technical assumptions made in the theorem, the property $\spb(A) < \spb(B)$ is a \emph{necessary} condition for domination between $(e^{tB})_{t \in [0,\infty)}$ and $(e^{tA})_{t \in [0,\infty)}$. A classical result of this type for positive semigroups can be found in \cite[Theorem~1.3]{AreBat92}.
	
	This is reminiscent of a classical theme in Perron--Frobenius theory where one considers operators $0 \le S \le T$ with the same spectral radius and proves that this implies $S = T$ under appropriate technical assumptions; see for instance \cite{Gao2013} and the references therein for an overview.
\end{remark}

Before we prove Theorem~\ref{thm:main-result-individual} it is certainly worthwhile to discuss its assumptions in a bit more detail:

\emph{Spectral theoretic assumptions:} The assumption that both spectra $\spec(A)$ and $\spec(B)$ are non-empty is rather mild and only required for technical reasons; besides, it excludes trivial cases such nilpotent semigroups, which are of course eventually dominated by any other semigroup. The assumption is for instance satisfied for most (though not all) differential operators occurring in applications and, in particular, for all self-adjoint operators on Hilbert spaces. The assumption that $\spec_\per(A)$ and $\spec_\per(B)$ consist of poles of the resolvents of $A$ and $B$, respectively, is for instance satisfied if $A$ and $B$ have compact resolvent -- a condition which is satisfied by a large class of differential operators on bounded domains in $\bbR^d$. We note that if $E$ is an $L^p$-space for $p \in [1,\infty)$, or more generally a Banach lattice with order continuous norm, then the smoothing conditions $e^{t_0A} E \subseteq E_u$ and $e^{t_0B}E \subseteq E_u$ imply that the semigroups $(e^{tA})_{t \in [0,\infty)}$ and $(e^{tB})_{t \in [0,\infty)}$ are eventually compact and thus, all spectral values of $A$ and $B$ are poles of their resolvents, respectively \cite[Corollary~2.5]{Daners2017}.

\emph{Smoothing assumptions:} Analyticity is an assumption which is met by many semigroups appearing in applications. In particular, if $E$ is an $L^2$-space and $A$ and $B$ are self-adjoint, then the semigroups generated by $A$ and $B$ are automatically analytic.

The assumptions $e^{t_0A} E \subseteq E_u$ and $e^{t_0B}E \subseteq E_u$ are a bit more subtle. In order to properly understand them, let us first focus on the semigroup $(e^{tB})_{t \in [0,\infty)}$ and let us discuss how $u$ is related to certain spectral properties of the generator $B$. Since $(e^{tB})_{t \in [0,\infty)}$ is assumed to be individually eventually positive, the spectral bound $\spb(B)$ is a spectral value of $B$ according to \cite[Theorem~7.6]{Daners2016}; thus $\spb(B)$ is, again by assumption, a pole of $\Res(\argument,B)$ and therefore even an eigenvalue of $B$ with a positive eigenvector $w$ (see \cite[Theorem~7.7]{Daners2016} or \cite[Corollary~3.3]{Daners2017}). Now we have, on the one hand, $w = e^{-t\spb(B)}e^{tB}w \gg_u 0$ for sufficiently large $t$ and, on the other hand, $w = e^{-t_0\spb(B)}e^{t_0B}w \in e^{t_0B}E \subseteq E_u$, so $u \gg_w 0$. This shows that
\begin{align*}
	c_1 w \le u \le c_2 w
\end{align*}
for two real numbers $c_1,c_2 > 0$. Hence, we may interpret $u$ as a sort of ``deformed'' eigenvector of $B$ for the eigenvalue $\spb(B)$ (see also \cite[Section~3]{DanersUNIF} for a closely related discussion). Note that a similar reasoning as above shows that $\spb(A)$ is a spectral value and a pole of the resolvent for $A$ with a positive eigenvector $v$. We also have $v \in E_u$, so $d_1 v \le u$ for an appropriate number $d_1 > 0$. However, we cannot bound $u$ from above by a multiple of $v$ since we did not assume the individual eventual positivity of $(e^{tA})_{t \in [0,\infty)}$ to be strong with respect to $u$.

Now we elaborate a bit more on the conditions $e^{t_0A}E \subseteq E_u$ and $e^{t_0B}E \subseteq E_u$ themselves. In a prototypical example we have $E = L^p(\Omega,\mu)$ for a finite measure space $(\Omega,\mu)$ and $u = \one_\Omega$. In this case, $E_u = L^\infty$, and the condition $e^{t_0A}E \subseteq E_u$ (and likewise for $e^{t_0B}$) can sometimes by verified by means of kernel estimates (see for instance \cite[Theorem~4.1]{DanersUNIF}). Alternatively, we note that analyticity of $(e^{tA})_{t \in [0,\infty)}$ implies that $e^{t A}E \subseteq D(A^\infty) := \bigcap_{n \in \bbN} D(A^n)$ for every $t>0$, so it suffices to show that $D(A^n) \subseteq E_u$ for at least one $n \in \bbN$ to conclude that $e^{t_0A} \subseteq E_u$ at some (even every) time $t_0 > 0$. The condition $D(A^n) \subseteq E_u$ often comes down to a Sobolev embedding theorem (see the subsection on \emph{The Laplace operator with non-local Robin boundary conditions} of \cite[Section~6]{DanGluKen16b} for several examples of this).

Recall from above that $u$ behaves similarly to the eigenvector $w$ of $B$. Hence, if $B$ is a differential operator on a bounded set $\Omega \subseteq \bbR^d$ subject to, say, Dirichlet boundary conditions, then $u$ vanishes on the boundary of $\Omega$. In order to prove that $D(A^n) \subseteq E_u$ for some $n \in \bbN$ one thus has to show that all functions in $D(A^n)$ vanish at least as fast as $u$ at the boundary of $\Omega$; we refer to \cite[Proposition~6.5]{DanGluKen16b} for an example where this is checked. Also see Subsection~\ref{subsec:elliptic-ops-with-dirichlet-bc} of the present article for a related example.

Finally, we discuss the case where $E$ is a space $C(K)$ of continuous complex-valued functions on a compact Hausdorff space $K$. Then the assumption that $u$ be a quasi-interior point of $E_+$ implies that $u \ge \varepsilon \one_K$ for a real number $\varepsilon > 0$, so we have $E_u = E$ and the assumptions $e^{t_0A}E \subseteq E_u$ and $e^{t_0B}E \subseteq E_u$ are trivial. We note in passing that this is essentially the reason why it is easier to characterise eventually strong positivity on $C(K)$-spaces (as done in \cite{Daners2016}) than on general Banach lattices (as done in \cite{DanGluKen16b}).

\emph{Positivity assumptions:} For the theory of eventually positive semigroups we refer to the series of papers \cite{Daners2016, DanGluKen16b, Daners2017, DanersPERT, DanersUNIF} and to \cite{GlueckDISS}; numerous examples of eventually positive semigroups can be found in \cite[Section~6]{Daners2016}, \cite[Section~6]{DanGluKen16b}, \cite[Section~4]{DanersUNIF} and \cite[Chapter~11]{GlueckDISS}. However, we point out that our characterisation of eventual domination above seems to be new even in the case that both semigroups $(e^{tA})_{t \in [0,\infty)}$ and $(e^{tB})_{t \in [0,\infty)}$ are positive.

\begin{proof}[Proof of Theorem~\ref{thm:main-result-individual}]
	There is no loss of generality in assuming throughout the proof that $\spb(B) = 0$.

	``(i) $\Rightarrow$ (iii)'' We first note that $\spb(A) \le \spb(B)$. Indeed, for $f \in E$, we can represent $\Res(\lambda,B)f$ as Laplace transform of the orbit $t \mapsto e^{tB}f$ whenever $\re \lambda > \spb(B)$; this is a consequence of the eventual positivity of $(e^{tB})_{t \in [0,\infty)}$, see \cite[Proposition~7.1]{Daners2016}. Hence, the eventual domination condition in~(i) implies, for $\re \lambda > \spb(B)$ and $f \ge 0$, that
	\begin{align*}
		\left(\int_0^t e^{-\lambda s}e^{sA} f \dx s\right)_{t \in [0,\infty)}
	\end{align*}
	is a Cauchy net in $E$ and thus convergent. Hence, $\lambda$ is in the resolvent set of $A$ (see e.g.\ \cite[Theorem~II.1.10(i)]{Engel2000}). So indeed, $\spb(A) \le \spb(B)$.
	
	Now, let $\spb(A) = \spb(B) = 0$. We show that this implies $(e^{tA})_{t \in [0,\infty)} = (e^{tB})_{t \in [0,\infty)}$, which contradicts the assumption that both semigroups are distinct.
	
	As the semigroup $(e^{tB})_{t \in [0,\infty)}$ is eventually positive and its generator $B$ has non-empty spectrum, it follows from \cite[Theorem~7.6]{Daners2016} that the spectral bound $\spb(B) = 0$ is a spectral value of $B$. According to \cite[Theorem~5.1]{Daners2017} our assumptions on the semigroup $(e^{tB})_{t \in [0,\infty)}$ now imply that the spectral projection $P$ that corresponds to the eigenvalue $0$ of $B$ satisfies $Pf \gg_u 0$ for all $0 < f \in E$. This in turn implies, according to \cite[Corollary~3.3]{DanGluKen16b}, that $0$ is a first order pole of the resolvent of $B$. 
	
	Hence, the range of $P$ coincides with the kernel of $B$ and the range of the dual operator $P'$ coincides with the kernel of the dual operator $B'$ of $B$. By using again \cite[Corollary~3.3]{DanGluKen16b} we see that both $\ker B$ and $\ker B'$ are one-dimensional, that the first of those spaces is spanned by a vector $v \gg_u 0$ and that the latter space is spanned by a strictly positive functional $\varphi \in E'$. We may choose $v$ and $\varphi$ such that $\langle \varphi, v\rangle = 1$ and such that $\norm{\varphi} = 1$, and then $P$ can be written in the form $P = \varphi \otimes v$.
	
	Since $(e^{tB})_{t \in [0,\infty)}$ is analytic, and thus eventually norm continuous, it follows from our spectral theoretic assumptions that $\spec_\per(B)$ is finite; consequently, the assumptions of \cite[Theorem~5.2]{DanGluKen16b} are satisfied, so this theorem implies that the semigroup $(e^{tB})_{t \in [0,\infty)}$ is bounded. Hence, the semigroup $(e^{tA})_{t \in [0,\infty)}$ is bounded, too, due to the eventual domination condition in assertion~(i). Therefore, the eventual positivity of $(e^{tA})_{t \in [0,\infty)}$ implies that the latter semigroup is \emph{individually asymptotically positive} in the sense of \cite[Definition~8.1(a)]{DanGluKen16b}. The spectral assumptions on $A$ and the eventual norm continuity of $(e^{tA})_{t \in [0,\infty)}$ (which is a consequence of the analyticity) imply that $A$ has only finitely many spectral values on the imaginary axis, so \cite[Theorem~8.3]{DanGluKen16b} now shows that the spectral bound $\spb(A) = 0$ of $A$ is a spectral value of $A$, that the corresponding spectral projection $Q$ is a positive operator and that $e^{tA}$ converges strongly to $Q$ as $t \to \infty$. 
	
	The eventual domination condition in~(i) implies that $P \ge Q \ge 0$. From this we conclude that $Q = QPQ$, since we have
	\begin{align*}
		0 \le QPQ - QQQ = Q(P-Q)Q \le P(P-Q)Q = PQ - PQ = 0;
	\end{align*}
	this argument is taken from \cite[Proposition~2.1.3]{Emelyanov2007}. In particular, the rank of $Q$ is dominated by the rank of $P$. As $P$ has rank $1$ and $Q$ is non-zero, we conclude that $Q$ also has rank $1$. Thus, $Q$ can be written in the form $Q = \psi \otimes w$, where $w \in E$ and $\psi \in E'$ are non-zero vectors. We note that $w$ is a fixed vector of $(e^{tA})_{t \in[0,\infty)}$ and $\psi$ is a fixed vector of the dual semigroup. Since $Q$ is positive, we can choose both $\psi$ and $w$ to be positive, too, and of course we can choose $\psi$ to have norm $1$. Note that $\langle \psi, w\rangle = 1$ since $Q$ is a projection.
	
	Our next goal is to show that $\varphi = \psi$. If we apply the inequality $Q' \le P'$ to the functional $\psi$ we obtain $\psi \le \langle \psi, v\rangle \varphi$. Hence, the functional $\langle \psi, v\rangle \varphi - \psi \in E'$ is positive, and one readily checks that this functional yields $0$ when tested against the vector $v$. Since $v$ is a quasi-interior point of $E_+$ (which follows from $v \gg_u 0$), this implies that actually $\langle \psi, v\rangle \varphi - \psi = 0$. Thus, the functionals $\varphi$ and $\psi$ are linearly dependent; since they are both positive and of norm $1$, we conclude that they actually coincide.
	
	We can use the identity $\varphi = \psi$ to finally show that the semigroups $(e^{tA})_{t \in [0,\infty)}$ and $(e^{tB})_{t \in [0,\infty)}$ coincide. To this end, let $0 < f \in E$ and consider any time $t \ge t_1$, where the $f$-dependent time $t_1 \in [0,\infty)$ is chosen as in assertion~(i). Then $e^{tB}f - e^{tA}f \ge 0$ and, moreover,
	\begin{align*}
		\langle \varphi, e^{tB}f - e^{tA}f \rangle = \langle (e^{tB})'\varphi,f \rangle - \langle (e^{tA})'\psi, f\rangle = \langle \varphi, f\rangle - \langle \psi, f\rangle = 0.
	\end{align*}
	Since the functional $\varphi$ is strictly positive, we conclude that $e^{tB}f - e^{tA}f = 0$. The analyticity of both semigroups now implies that we actually have $e^{tB}f = e^{tA}f$ for all times $t \in (0,\infty)$. Since $f > 0$ was arbitrary and since the positive cone $E_+$ spans $E$ we conclude that both semigroups coincide.
	
	``(iii) $\Rightarrow$ (ii)'' As mentioned in the proof of the previous implication, the assumptions on $B$ imply that \cite[Theorem~5.2]{DanGluKen16b} is applicable. We now use that part~(iii) of this theorem yields strong convergence of $e^{tB}$ to an operator $P$ as $t \to \infty$; by the same theorem, this operator $P$ satisfies $Pf \gg_u 0$ for every $f > 0$. Moreover, it follows from $e^{t_0B}E \subseteq E_u$ and from the closed graph theorem that $e^{t_0B}$ is a bounded operator from the Banach space $E$ to the Banach space $E_u$, where the latter is endowed with the gauge norm $\norm{\argument}_u$ with respect to $u$. Now fix $0 < f \in E$. We then obtain
	\begin{align*}
	\begin{split}
		 \norm{e^{tB}f - Pf}_u& = \norm{e^{t_0B}e^{(t-t_0)B}f - e^{t_0B}Pf}_u \\
		& \le \norm{e^{t_0B}}_{E \to E_u} \norm{e^{(t-t_0)B}f - Pf}_E \to 0 \quad \text{as } t \to \infty;
\end{split}	
	\end{align*}
	for the first equality we used that the range of $P$ coincides with the fixed space of $(e^{tB})_{t \in [0,\infty)}$. We have $Pf \ge \varepsilon u$ for some $\varepsilon > 0$, and we have just shown that $\modulus{e^{tB}f - Pf} \le \frac{\varepsilon}{2} u$ for all sufficiently large $t$, say $t \ge t_2$. Since the semigroup $(e^{tB})_{t \ge 0}$ is real, we conclude that $e^{tB}f \ge \frac{\varepsilon}{2}u$ for all $t \ge t_2$.
	
	On the other hand, the generator $A$ of the analytic semigroup $(e^{tA})_{t \in [0,\infty)}$ has spectral bound $\spb(A) < \spb(B) = 0$, so $e^{tA}$ converges to $0$ with respect to the operator norm as $t\to \infty$ (see e.g.\ \cite[Corollary~IV.3.12]{Engel2000} or \cite[Theorem~5.1.12]{Arendt2011}). Similarly as above we now use that $e^{t_0A}E \subseteq E_u$, which implies that $e^{t_0A}$ is a bounded operator from $E$ to $E_u$ due to the closed graph theorem. Hence,
	\begin{align*}
		\norm{e^{tA}}_{E \to E_u} \le \norm{e^{t_0A}}_{E \to E_u} \norm{e^{(t-t_0)A}}_{E \to E} \to 0 \quad \text{as } t \to \infty.
	\end{align*}
	
	Therefore, there exists a time $t_3 \in [0,\infty)$ such that $\norm{e^{tA}f}_u \le \frac{\varepsilon}{4}$ for all $t \ge t_3$, and thus we have $\modulus{e^{tA}f} \le \frac{\varepsilon}{4}u$ for all $t \ge t_3$. Using that the semigroup $(e^{tA})_{t \in [0,\infty)}$ is real, we thus conclude that $e^{tA}f \le \frac{\varepsilon}{4}u$ for all $t \ge t_3$. 
	
	Consequently, we have $e^{tB}f - e^{tA}f \ge \frac{\varepsilon}{2}u - \frac{\varepsilon}{4}u \gg_u 0$ for all $t \ge t_1 := \max\{t_2,t_3\}$. This proves assertion~(ii).
	
	``(ii) $\Rightarrow$ (i)'' This implication is obvious.
\end{proof}

\begin{remark}
	In the proof of Theorem~\ref{thm:main-result-individual}, the analyticity assumption on the semigroups is only needed in its full strength to employ the identity theorem for analytic mappings at the end of the proof of implication ``(i) $\Rightarrow$ (iii)''. We do not know whether this implication remains true if the analyticity assumption is weakened (for instance, to eventual norm continuity). 
	
	The implication ``(iii) $\Rightarrow$ (ii)'' remains true if we only assume that both semigroups $(e^{tA})_{t \ge 0}$ and $(e^{tB})_{t \ge 0}$ are eventually norm continuous rather than analytic.

	The assumptions $e^{t_0A} E \subseteq E_u$ and $e^{t_0B}E \subseteq E_u$ in Theorem~\ref{thm:main-result-individual} can be omitted for the implication ``(i) $\Rightarrow$ (iii)''. We do not know whether the assumption $e^{t_0B}E \subseteq E_u$ is really needed for the converse implication ``(iii) $\Rightarrow$ (ii)'' to hold; but the following example shows that at least the assumption $e^{t_0A}E \subseteq E_u$ cannot be dropped, in general.
\end{remark}

\begin{example}
	Let $E = L^2(0,\pi)$, let $A$ be the Neumann Laplace operator on $E$ and let $B = \Delta + 2\id_E$ with Dirichlet boundary conditions. Both semigroups $(e^{tA})_{t \in [0,\infty)}$ and $(e^{tB})_{t \in [0,\infty)}$ are positive (thus, in particular, real) and analytic. Moreover, the peripheral spectra $\spec_\per(A) = \{0\}$ and $\spec_\per(B) = \{1\}$ consist of poles of the resolvents of $A$ and $B$, respectively.
	
	Consider the function $\sin \in E$; we have $e^{tB}\sin = e^t \sin$ for each $t \in [0,\infty)$ and, on the other hand, $e^{tA}\sin \gg_{\one} 0$ for each $t \in (0,\infty)$. Hence, $e^{tB}\sin$ does not dominate $e^{tA} \sin$ for any $t \in (0,\infty)$, in spite of the fact that $\spb(B) = 1 > 0 = \spb(A)$. Hence, there cannot be any quasi-interior point $u$ of $E_+$ such that the assumptions of Theorem~\ref{thm:main-result-individual} are satisfied.
	
	It is illuminating to discuss this in a bit more detail for two particular quasi-interior points:
	\begin{itemize}
		\item Let $u = \one$. Then all assumptions of Theorem~\ref{thm:main-result-individual} are satisfied except that $(e^{tB})_{t \in [0,\infty)}$ is not individually eventually strongly positive with respect to $\one$.
		
		\item Let $u = \sin$. Then all assumptions of Theorem~\ref{thm:main-result-individual} are satisfied except that we do not have $e^{t_0A}E \subseteq E_u$ for any $t_0 \in [0,\infty)$.
	\end{itemize}
\end{example}

Under the assumptions of Theorem~\ref{thm:main-result-individual}, if $\spb(A) = \spb(B)$ and $t_0 \in [0,\infty)$, then there exist vectors $0 < f_1,f_2 \in E$ and times $t_1,t_2 \ge t_0$ such that $e^{t_1 A}f_1>e^{t_1 B}f_1$ and $e^{t_2 B}f_2>e^{t_2 A}f_2$. We stress that this does \textit{not} imply lack of eventual domination for individual orbits. Indeed, it may still happen that, for each $f \ge 0$, either the orbit $(e^{tA} f)_{t \ge 0}$ eventually dominates the orbit $(e^{tB}f)_{t \ge 0}$ or vice versa. Let us illustrate this by means of the following example.

\begin{example} \label{ex:cone-splits}
	On the Banach lattice $E = \bbC^2$, let $u = (1, 1)^T$, and consider the projections
	\begin{align*}
		P = \frac{1}{3}
		\begin{pmatrix}
			1 & 2 \\ 1 & 2
		\end{pmatrix}
		\quad \text{and} \quad
		Q = \frac{1}{3}
		\begin{pmatrix}
			2 & 1 \\ 2 & 1
		\end{pmatrix}.
	\end{align*}
	For each vector $0 \le x \in \bbC^2$ we have $Px \ge Qx$ iff $x_2 \ge x_1$ and $Px \le Qx$ iff $x_2 \le x_1$. Now consider the operator semigroups $(e^{tA})_{t \ge 0}$ and $(e^{tB})_{t \ge 0}$ given by
	\begin{align*}
		& e^{tA} := P + e^{-t}(\id-P) = e^{-t}\id + (1-e^{-t})P, \\
		\text{and} \qquad & e^{tB} := Q + e^{-t}(\id-Q) = e^{-t}\id + (1-e^{-t})Q
	\end{align*}
	with generators $A = P - \id$ and $B = Q - \id$, respectively. Both semigroups are strongly positive (in the sense that all entries of $e^{tA}$ and $e^{tB}$ are strictly positive for each $t > 0$), and both generators have spectral bound $0$. Let us show that $E_+$ can be written as the union of two (non-disjoint) cones $C_1,C_2$ such $(e^{tA})_{t\ge 0}$ dominates $(e^{tB})_{t\ge 0}$ on $C_1$, and vice versa on $C_2$.
	
	Indeed, let $0 \le x \in \bbC^2$. If $x_2 \ge x_1$, then $Px \ge Qx$ and thus, $e^{tA}x = e^{-t}e^{tP}x \ge e^{-t}e^{tQ}x = e^{tB}x$ for all $t \ge 0$, so the orbit of $x$ under the semigroup $(e^{tA})_{t \ge 0}$ dominates the orbit of $x$ under the semigroup $(e^{tB})_{t \ge 0}$. If instead $x_2 \le x_1$, then we conversely obtain $e^{tA}x \le e^{tB}x$ for all $t \ge 0$.
\end{example}

The above example is somewhat special in the sense that the positive cone can be written as the union of two cones $C_1$ and $C_2$ such that $(e^{tA})_{t\ge 0}$ eventually dominates (in this particular example even dominates) $(e^{tB})_{t\ge 0}$ on $C_1$, and vice versa on $C_2$. This is a very special situation that cannot be expected in general. Here is a simple counterexample.

\begin{example} \label{ex:cone-does-not-split}
	Consider the orthonormal basis $(u_1,u_2,u_3)$ of $\bbC^3$ given by
	\begin{align*}
		u_1 = \frac{1}{\sqrt{3}}
		\begin{pmatrix}
			1 \\ 1 \\ 1
		\end{pmatrix},
		\quad
		u_2 = \frac{1}{\sqrt{2}}
		\begin{pmatrix}
			-1 \\ 0 \\ 1
		\end{pmatrix},
		\quad 
		u_3 = \frac{1}{\sqrt{6}}
		\begin{pmatrix}
			1 \\ -2 \\ 1
		\end{pmatrix}
	\end{align*}
	und let $U \in \bbC^{3 \times 3}$ be the matrix with columns $u_1$, $u_2$ and $u_3$. Then the operators
	\begin{align*}
		A :=
		U
		\begin{pmatrix}
			0 & 0 & 0 \\
			0 & -1 & 0 \\
			0 & 0 & -1
		\end{pmatrix}
		U^{-1}
		\quad
		\text{and}
		\quad
		B :=
		U
		\begin{pmatrix}
			0 & 0 & 0 \\
			0 & -1 & 1 \\
			0 & -1 & -1
		\end{pmatrix}
		U^{-1}
	\end{align*}
	both have spectral bound $0$; more precisely, 
	\begin{align*}
		\spec(A) = \{0,-1\} 
		\quad \text{and} \quad 
		\spec(B) = \{0,i-1,-i-1\}.
	\end{align*}
	Both semigroups $(e^{tA})_{t \in [0,\infty)}$ and $(e^{tB})_{t \in [0,\infty)}$ are uniformly eventually strongly positive with respect to the vector $u_1$ since both semigroups converge to the projection $u_1u_1^T$; we conclude from Theorem~\ref{thm:main-result-individual} that $(e^{tB})_{t \in [0,\infty)}$ does not eventually dominate $(e^{tA})_{t \in [0,\infty)}$, nor vice versa. 
	
	Moreover, we do not have such a splitting of the positive cone $(\bbC^3)_+$ as in Example~\ref{ex:cone-splits}. To see this, let $x := 2u_1 + u_2 \ge 0$. Then $e^{tA}x = 2u_1 + e^{-t}u_2$ and $e^{tB}x = 2u_1 + e^{-t}(\cos(t) u_2 - \sin(t) u_3)$ for all $t \in [0,\infty)$. In particular, we have $e^{tB}x = 2u_1 + e^{-t} u_3$ whenever $t \in 2\pi\bbZ + \frac{3}{2}\pi$, so we neither have $e^{tB}x \ge e^{tA}x$ nor vice versa for those particular times $t$.
\end{example}

\subsection{Uniform eventual domination}

Our second main result gives criteria for eventual domination where the time $t_1$ from which on $e^{tB}f$ dominates $e^{tA}f$ does not depend on $f$. We formulate the result in the setting of self-adjoint semigroups on Hilbert spaces. 

\begin{theorem} \label{thm:main-result-uniform}
	Let $L^2 := L^2(\Omega,\mu)$ for a $\sigma$-finite measure space $(\Omega,\mu)$ and let $u \in (L^2)_+$ be a function which is strictly positive almost everywhere.
	Consider two distinct real $C_0$-semigroups $(e^{tA})_{t \in [0,\infty)}$ and $(e^{tB})_{t \in [0,\infty)}$ on $L^2$ which satisfy the following assumptions: 
	\begin{itemize}
		\item \emph{Self-adjointness assumptions:} Both operators $A$ and $B$ are self-adjoint;
		\item \emph{Smoothing assumptions:} There exists a time $t_0 \in [0,\infty)$ such that $e^{t_0A}L^2 \subseteq (L^2)_u$ and $e^{t_0B} L^2 \subseteq (L^2)_u$; and
		\item \emph{Positivity assumptions:} The semigroup $(e^{tA})_{t \in [0,\infty)}$ is uniformly eventually positive, and the semigroup $(e^{tB})_{t \in [0,\infty)}$ is uniformly (equivalently: individually) eventually strongly positive with respect to $u$.
	\end{itemize}
	Then the following assertions are equivalent:
	\begin{enumerate}[\upshape (i)]
		\item For all $0 < f \in L^2$ there exists a time $t_1 \in [0,\infty)$ such that $e^{tB}f \ge e^{tA}f \ge 0$ for all $t \ge t_1$.
		\item There exists a time $t_1 \in [0,\infty)$ and a number $\delta > 0$ such that $e^{tB} \ge e^{tA} + \delta u \otimes u \ge e^{tA}$ for all $t \ge t_1$.
		\item We have $\spb(B) > \spb(A)$.
	\end{enumerate}
\end{theorem}

In the situation of Theorem~\ref{thm:main-result-uniform} we can also add a fourth assertion to the list of equivalent statements:
\begin{enumerate}[\upshape (iv)]
	\item \textit{There exists a time $t_1 \in [0,\infty)$ such that $e^{tB} \ge e^{tA}$ for all $t \ge t_1$.}
\end{enumerate}
Formally, this assertion is stronger than~(i) and weaker than~(ii), so it is -- under the assumptions of Theorem~\ref{thm:main-result-uniform} -- equivalent to (i)--(iii). We use the terminology ``$(e^{tB})_{t \in [0,\infty)}$ \emph{uniformly eventually dominates} $(e^{tA})_{t \in [0,\infty)}$'' to describe the situation in~(iv).
	
The fact mentioned in the positivity assumptions of Theorem~\ref{thm:main-result-uniform} that uniform and individual eventual strong positivity with respect to $u$ are equivalent for $(e^{tB})_{t \in [0,\infty)}$ is a consequence of the self-adjointness of $B$; see \cite[Corollary~3.5]{DanersUNIF} (and see also \cite[Theorem~10.2.1]{GlueckDISS} for a closely related result). Note that assertions~(i) and~(iii) above are the same as in Theorem~\ref{thm:main-result-individual}, but assertion (ii) is stronger than the corresponding assertion in Theorem~\ref{thm:main-result-individual}.

For the proof of Theorem~\ref{thm:main-result-uniform} we adapt a method that is used to characterise uniform eventual strong positivity of self-adjoint $C_0$-semigroups in \cite[Theorem~10.2.1]{GlueckDISS}. We suspect that, by using techniques from \cite{DanersUNIF}, it should be possible to prove uniform eventual domination results in the spirit of Theorem~\ref{thm:main-result-uniform} above for non-self-adjoint semigroups; however, such results would contain even more technical assumptions involving the dual semigroups of $(e^{tA})_{t\in [0,\infty)}$ and $(e^{tB})_{t \in [0,\infty)}$. Thus, in order to keep our results as comprehensible as possible -- and since we are mainly interested in the case of self-adjoint semigroups in the subsequent section on applications -- we restrict ourselves to the self-adjoint case here.

\begin{proof}[Proof of Theorem~\ref{thm:main-result-uniform}]
	``(ii)~$\Rightarrow$~(i)'' This implication is obvious.
	
	``(i)~$\Rightarrow$~(iii)'' We claim that this implication follows from Theorem~\ref{thm:main-result-individual}; to see this we have to check that all assumptions of this theorem are satisfied. Since $A$ and $B$ are self-adjoint, their spectra are non-empty, and the semigroups $(e^{tA})_{t \in [0,\infty)}$ and $(e^{tB})_{t \in [0,\infty)}$ are analytic. Moreover, since $L^2$ is reflexive, it follows from the smoothing assumptions $e^{t_0A}L^2 \subseteq (L^2)_u$ and $e^{t_0B}L^2 \subseteq (L^2)_u$ that the operators $e^{t_0A}$ and $e^{t_0B}$ are compact \cite[Theorem~2.3(ii)]{Daners2017}. Hence, the semigroups generated by $A$ and $B$ are eventually compact, so all spectral values of $A$ and $B$ are poles of the resolvents of $A$ and $B$, respectively \cite[Corollary~V.3.2(i)]{Engel2000}. Thus, Theorem~\ref{thm:main-result-individual} is applicable.
	
	``(iii)~$\Rightarrow$~(ii)'' There is no loss of generality in assuming that $\spb(B) = 0$, so let $\spb(B) = 0 > \spb(A)$. For the rest of the proof we also assume, for the sake of notational simplicity, that $L^2$ is infinite-dimensional. Over finite-dimensional spaces, we can use the same proof but we have to replace all the infinite eigenvalue expansions in the following with finite sums.
	
	We have already noted above that the semigroups $(e^{tA})_{t \in [0,\infty)}$ and $(e^{tB})_{t \in [0,\infty)}$ are eventually compact and analytic; hence, they are immediately compact (apply the identity theorem for analytic functions in the Calkin algebra over $L^2$) and thus, $A$ and $B$ have compact resolvent according to \cite[Theorem~II.4.29]{Engel2000}. In particular, the domain $D(B)$ (endowed with the graph norm) embeds compactly into $L^2$. Thus, the unit ball of $D(B)$ is a relatively compact subset of $L^2$ and thus separable with respect to the norm on $L^2$. As the span of the unit ball of $D(B)$ is dense in $L^2$ we conclude that $L^2$ is separable. 
	
	Due to the self-adjointness of $A$ and $B$ and, again, due to the compactness of their resolvents, it now follows from the spectral theorem that the eigenvalues of $A$ and $B$ are given by sequences $0 > -\lambda_1 \ge - \lambda_2 \ge \dots$ and $0 = -\mu_0 \ge -\mu_1 \ge -\mu_2 \ge \dots$, respectively, which both converge to $-\infty$; moreover, we can find orthonormal bases $(e_n)_{n \in \bbN}$ and $(f_n)_{n \in \bbN_0}$ of $L^2$ such that each vector $e_n$ is an eigenvector of $A$ for the eigenvalue $-\lambda_n$ and such that each vector $f_n$ is an eigenvector of $B$ for the eigenvalue $-\mu_n$. Note that all vectors $e_n$ and $f_n$ can be chosen to be real-valued functions since the operators $A$ and $B$ are real. Moreover, due to the eventual strong positivity of $(e^{tB})_{t \in [0,\infty)}$ with respect to $u$ it follows, for instance, from \cite[Corollary~3.5]{DanersUNIF} that the eigenspace $\ker B$ is one-dimensional and spanned by a function $w \gg_u 0$; thus, we have $0 = -\mu_0 > -\mu_1$, and the eigenvector $f_0$ can be chosen such that $f_0 \gg_u 0$, i.e., we have $f_0 \ge cu$ for an appropriate number $c > 0$.
	
	The operators $e^{t_0A}$ and $e^{t_0B}$ are continuous from $L^2$ to $(L^2)_u$ as a consequence of the closed graph theorem; let $M$ denote the maximum of their operator norms $\norm{e^{t_0A}}_{L^2 \to (L^2)_u}$ and $\norm{e^{t_0B}}_{L^2 \to (L^2)_u}$. For every $n \in \bbN$ we have $\norm{f_n}_u = e^{t_0\mu_n}\norm{e^{t_0B}f_n}_u$, and likewise for $\norm{e_n}_u$. So
	\begin{align*}
		\norm{f_n}_u \le M e^{t_0 \mu_n} \qquad \text{and} \qquad \norm{e_n}_u \le M e^{t_0 \lambda_n} \qquad \text{for all } n \in \bbN,
	\end{align*}
	and thus, $\modulus{f_n} \le Me^{t_0 \mu_n} u$ and $\modulus{e_n} \le Me^{t_0 \lambda_n}u$ for all $n \in \bbN$. Now fix $0 \le g \in L^2$. For every $t \in [0,\infty)$,
	\begin{align}
		\label{eq:difference-to-be-positive}
		e^{tB}g - e^{tA}g = \langle g,f_0\rangle f_0 + \sum_{n=1}^\infty \left( e^{-t\mu_n} \langle g,f_n\rangle f_n - e^{-t\lambda_n} \langle g,e_n\rangle e_n \right).
	\end{align}
	Let us show that, for sufficiently large $t$, the series on the right is even absolutely convergent with respect to the gauge norm $\norm{\argument}_u$: we observe that
	\begin{align}
		& \sum_{n=1}^\infty \norm{e^{-t\mu_n} \langle g,f_n\rangle f_n - e^{-t\lambda_n} \langle g,e_n\rangle e_n}_u \nonumber \\
		& \le \sum_{n=1}^\infty \left[ e^{-t\mu_n} \langle g,\modulus{f_n}\rangle \norm{f_n}_u + e^{-t\lambda_n} \langle g, \modulus{e_n}\rangle \norm{e_n}_u \right] \nonumber \\
		& \le M^2 \langle g,u \rangle \sum_{n=1}^\infty \left[ e^{(-t + 2t_0)\mu_n} + e^{(-t + 2t_0)\lambda_n} \right]
		\label{eq:series-gauge-norm-estimate}
	\end{align}
	for every $t \in [0,\infty)$. Now we note that $e^{t_0A}$ and $e^{t_0B}$ are Hilbert--Schmidt operators. Indeed, it follows from the assumption $e^{t_0B}L^2 \subseteq (L^2)_u$ that the restriction of $e^{t_0B}$ to the real Hilbert space of all real-valued functions on $L^2$ is a so-called \emph{majorising} operator, see \cite[Definition~IV.3.1]{Schaefer1974}; hence, this operator is Hilbert--Schmidt according to \cite[Theorem~IV.6.9]{Schaefer1974}, which in turn implies that the operator $e^{t_0B}$ on the complex space $L^2$ is Hilbert--Schmidt. By the same reasoning we can see that $e^{t_0A}$ is Hilbert--Schmidt.
	
Hence, we know that the eigenvalues of $e^{t_0A}$ and $e^{t_0B}$ are square summable, i.e., we have $\sum_{n=1}^\infty e^{-2t_0 \lambda_n} < \infty$ and $\sum_{n=0}^\infty e^{-2t_0 \mu_n} < \infty$. In particular, the series in~\eqref{eq:series-gauge-norm-estimate} is finite for all $t \ge 4t_0$; moreover, the number $\sum_{n=1}^\infty \left[ e^{(-t + 2t_0)\mu_n} + e^{(-t + 2t_0)\lambda_n}\right]$ even converges to $0$ as $t\to \infty$ as a consequence of the dominated convergence theorem, so we can find a time $t_1 \ge 4t_0$ such that
	\begin{align*}
		\sum_{n=1}^\infty \left[ e^{(-t + 2t_0)\mu_n} + e^{(-t + 2t_0)\lambda_n}\right] \le \frac{c^2}{2M^2}
	\end{align*}
	for all $t \ge t_1$. We have thus proved that, for all $t \ge 4t_0$, the series in~\eqref{eq:difference-to-be-positive} is absolutely convergent with respect to the gauge norm in $(L^2)_u$ and that its value is thus an element of $(L^2)_u$; moreover, the gauge norm of this value can by estimated as
	\begin{align*}
		\norm{\sum_{n=1}^\infty \left( e^{-t\mu_n} \langle g,f_n\rangle f_n + e^{-t\lambda_n} \langle g,e_n\rangle e_n \right)}_u \le \frac{c^2}{2} \langle g,u\rangle.
	\end{align*}
	for all $t \ge t_1$. Thus, we conclude from~\eqref{eq:difference-to-be-positive} that, for all $t \ge t_1$, 
	\begin{align*}
		e^{tB}g - e^{tA}g \ge \langle g, f_0\rangle f_0 - \frac{c^2}{2} \langle g,u\rangle u \ge c^2 \langle g,u\rangle u - \frac{c^2}{2} \langle g,u\rangle u = \frac{c^2}{2} \langle g,u\rangle u.
	\end{align*}
	Since $t_1$ does not depend on $g$, this proves~(ii) for $\delta = c^2/2$.
\end{proof}

\section{Applications} \label{section:applications}

In this section we discuss several examples where our results can be applied. We point out once again that, to the best of our knowledge, Theorems~\ref{thm:main-result-individual} and~\ref{thm:main-result-uniform} are new even if both semigroups $(e^{tA})_{t \in [0,\infty)}$ and $(e^{tB})_{t \in [0,\infty)}$ are positive. Thus, we also consider examples of positive semigroups in this section.

\subsection{The finite dimensional case}

In the related field of eventual positivity, a large part of the literature is devoted to the study of the finite-dimensional case. For detailed references, we refer for instance to the introduction of the article \cite{Shakeri2017}, and also to the introduction of \cite{Glueck2017} and to \cite[Section~6.4]{GlueckDISS}.

As a connection to the finite-dimensional literature, let us explicitly formulate Theorem~\ref{thm:main-result-individual} in the special case of matrices. We will use the fact that in finite dimensions individual and uniform eventual positivity are equivalent, and -- as can easily be seen -- the same is true for individual and uniform eventual domination. Hence, we use the term \emph{eventual domination} synonymously with \emph{individual eventual domination} and with \emph{uniform eventual domination} in finite dimensions.

\begin{theorem} \label{thm:matrix}
Let $A,B$ be two distinct $n\times n$-matrices with real entries. Assume that:
	\begin{enumerate}[(1)]
		\item The semigroup $(e^{tA})_{t \ge 0}$ is individually (equivalently: uniformly) eventually positive.
		\item $s(B)$ is a dominant and geometrically simple eigenvalue of $B$; moreover, the eigenspaces $\ker(s(B) - B)$ and $\ker(s(B) - B^T)$ contain vectors $x$ and $y$, respectively, such that $x_i > 0$ and $y_i > 0$ for all $i = 1,\dots,n$.
	\end{enumerate}
	Then the following assertions are equivalent:
	\begin{enumerate}[\upshape (i)]
		\item There exists a time $t_1 \in [0,\infty)$ such that $(e^{tB}x)_i \ge (e^{tA}x)_i$ for all $t \ge t_1$, all $0 < x \in \R^n$ and all $i=1,\ldots,n$ (i.e., $(e^{tB})_{t \in [0,\infty)}$ eventually dominates $(e^{tA})_{t \in [0,\infty)}$).
		\item There exists a time $t_1 \in [0,\infty)$ such that $(e^{tB}x)_i > (e^{tA}x)_i$ for all $t \ge t_1$, all $0 < x \in \R^n$ and all $i=1,\ldots,n$.
		\item We have $\spb(B) > \spb(A)$.
	\end{enumerate}
\end{theorem}

Note that Assumption~(2) is equivalent to eventual strong positivity of the semigroup (with respect to the vector $(1,\dots,1)$) by~\cite[Theorem~11.1.2]{GlueckDISS}; the assumption is in particular satisfied if $B$ generates a positive irreducible semigroup. Assumption~(1) is for instance satisfied if $A$ generates a positive semigroup (i.e., if all off-diagonal entries of $A$ are $\ge 0$).

\subsection{The Laplacian on finite graphs}

Let $A,B$ be two real $n\times n$ matrices that generate positive semigroups. A straightforward application of Ouhabaz' criterion in Proposition~\ref{prop:ouh-crit} shows that $(e^{tB})_{t\ge 0}$ dominates $(e^{tA})_{t\ge 0}$ if and only if $a_{ij}\le b_{ij}$ for all $i,j=1,\ldots,n$. 
As an immediate consequence, if $\mG=(\mV,\mE)$ is a finite connected undirected graph and $\mG'$ is a further graph with same vertex set $\mV'=\mV$, then the semigroup generated by the adjacency matrix of $\mG$ dominates the semigroup generated by the adjacency matrix of $\mG'$ if and only if $\mG'$ is a subgraph of $\mG$. An analogous assertion holds in the case of directed graphs $\mD=(\mV,\vec{\mE})$.

This is no more true if the semigroup generated by the discrete Laplacian is considered, instead. We will now show, as an application of Theorem~\ref{thm:matrix}, that we do not even have eventual domination in this case. Let us consider the advection matrix $\vec{\mathcal L}^\mD$ (sometimes called: ``directed discrete Laplacian'') of a strongly connected directed graph $\mD$, defined as in~\cite[\S~2.1.6]{Mug14}. Such advection matrices are real and satisfy the assumptions of Theorem~\ref{thm:matrix}: indeed, they are positive and irreducible~\cite[Lemma~4.57]{Mug14}, and it is known~\cite[Corollary~2.22]{Mug14} that all their eigenvalues have non-negative real part. Because furthermore ${\bf 1}$ lies in the kernel of any advection matrix and any Laplace matrix, we deduce the following.

\begin{prop}
	\begin{enumerate}[(1)]
		\item If $\mD_1=(\mV,\vec{\mE}_1)$, $\mD_2=(\mV,\vec{\mE}_2)$ are two distinct finite and strongly connected directed graphs, then neither does $(e^{-t\vec{\mathcal L}^{\mD_1}})_{t \in [0,\infty)}$ eventually dominate $(e^{-t\vec{\mathcal L}^{\mD_2}})_{t \in [0,\infty)}$, nor vice versa.

		\item If $\mG_1=(\mV,\mE_1)$, $\mG_2=(\mV,\mE_2)$ are two distinct finite and connected undirected graphs, then neither does $(e^{-t\mathcal L^{\mG_1}})_{t \in [0,\infty)}$ eventually dominate $(e^{-t\mathcal L^{\mG_2}})_{t \in [0,\infty)}$, nor vice versa.

		\item Let $\mG=(\mV,\mE)$ be a finite and connected graph and $\mD=(\mV,\vec{\mE})$ any strongly connected orientation of $\mG$. Then neither does $(e^{-t\mathcal L^\mG})_{t \in [0,\infty)}$ eventually dominate $(e^{-t\vec{\mathcal L}^{\mD}})_{t \in [0,\infty)}$, nor vice versa.
	\end{enumerate}
\end{prop}

\subsection{The one-dimensional Laplace operator with mixed and with periodic boundary conditions} \label{subsec:1-D-laplace-mixed-periodic}

Here, we briefly revisit the example from Subsection~\ref{subsection:heat-equation-different-boundary-conditions}: let $\Delta^M$ denote the Laplace operator on $L^2(0,1)$ with domain
\begin{align*}
	D(\Delta^M) = \{u \in H^2(0,1): \; u(0) = 0 \text{ and } u_x(1) = 0\}
\end{align*}
(i.e., we have mixed Dirichlet and Neumann boundary conditions), and let $\Delta^P$ denote the Laplace operator on $L^2(0,1)$ with domain
\begin{align*}
	D(\Delta^P) = \{u \in H^2(0,1): \; u(0) = u(1) \text{ and } u_x(0) = u_x(1)\}
\end{align*}
(i.e., we have periodic boundary conditions). Then $\spb(\Delta^M) =-\frac{\pi^2}{4}<0=\spb(\Delta^P)$. Let $u := \one \in L^2(0,1)$ denote the constant function with value $1$. Then the principal ideal $(L^2(0,1))_u$ is the space $L^\infty(0,1)$ and all assumptions of Theorem~\ref{thm:main-result-uniform} are satisfied. Hence, we obtain the following result:

\begin{prop}
	The semigroup $(e^{t\Delta^P})_{t \ge 0}$ does not dominate the semigroup $(e^{t\Delta^M})_{t \ge 0}$, but we have uniform eventual domination in the sense that $e^{t\Delta^P} \ge e^{t \Delta^M}$ for all sufficiently large times $t$.
\end{prop}

\subsection{The one-dimensional Laplace operator with Dirichlet and with non-local boundary conditions} \label{subsec:1-D-laplace-dirichlet-nonlocal}

In this subsection, we consider the Dirichlet Laplace operator $\Delta^D$ on $L^2(0,1)$ with domain
\begin{align*}
	D(\Delta^D) = \{u \in H^2(0,1): \; u(0) = u(1) = 0\}
\end{align*}
and a Laplace operator $\Delta^{NL}$ on $L^2(0,1)$ with non-local boundary conditions whose domain is given by
\begin{align*}
	D(\Delta^{NL}) = \{u \in H^2(0,1): \; u_x(0) = -u_x(1) = u(0) + u(1)\}.
\end{align*}
This is an interesting example of an operator with non-local boundary conditions since $-\Delta^{NL}$ is associated with the form 
\[
	(u,v)\mapsto\int_0^1 u_x(x)\overline{v_x(x)}\dx x + 
	\begin{pmatrix}
		u(0) & u(1)
	\end{pmatrix}
	\begin{pmatrix}
		1 & 1 \\
		1 & 1
	\end{pmatrix}
	\begin{pmatrix}
		\overline{v}(0) \\ \overline{v}(1)
	\end{pmatrix}
\]
on $H^1(0,1)$. The operator $-\Delta^{NL}$ has already occurred on several occassions in the literature as an example for eventual positivity; see \cite[Theorem~6.11]{DanGluKen16b}, \cite[Theorem~11.7.3]{GlueckDISS} and \cite[Example~4.10]{DanersPERT}. In fact, it follows from \cite[Theorem~11.7.3]{GlueckDISS} that the semigroup $(e^{t\Delta^{NL}})_{t \ge 0}$ on $L^2(0,1)$ is not positive, but uniformly eventually strongly positive with respect to $u := \one$. The non-positivity of this semigroup is also discussed in detail in \cite[Section~3]{Akhlil2018}. Now we show that the semigroup eventually dominates the Dirichlet Laplace semigroup $(e^{t\Delta^D})_{t \ge 0}$. We find this example particularly interesting since the dominated semigroup is positive, while the dominating semigroup is only eventually positive.

\begin{theorem}
	There exists a time $t_1 > 0$ such that $e^{t\Delta^{NL}} \ge e^{t\Delta^D}$ for all $t \ge t_1$.
\end{theorem}
\begin{proof}
	We first note that all assumptions of Theorem~\ref{thm:main-result-uniform} are satisfied for $u = \one$ (the smoothing assumption follows from the fact that both semigroups map into $H^1(0,1) \subseteq L^\infty(0,1)$). So it only remains to show that $\spb(\Delta^D) < \spb(\Delta^{NL})$.
	
	It is well-known (and easy to check) that $\spb(\Delta^D) = -\pi^2$. On the other hand, the mapping $(0,1) \ni x \mapsto \cos(\pi x) \in \bbR$ is contained in $D(\Delta^{NL})$ and is thus an eigenvector of $\Delta^{NL}$ for the eigenvalue $-\pi^2$. This shows that $\spb(\Delta^{NL}) \ge -\pi^2$. To see that this inequality is actually strict, we use once again that the semigroup $(e^{t\Delta^{NL}})_{t \ge 0}$ is uniformly eventually strongly positive with respect to $u = \one$ \cite[Theorem~11.7.3]{GlueckDISS}; according to \cite[Theorem~5.2 and Corollary~3.3]{DanGluKen16b} this implies that $\spb(\Delta^{NL})$ is an eigenvalue of $\Delta^{NL}$ whose eigenspace is spanned by a vector which is positive almost everywhere. Hence, this eigenspace cannot contain the function $(0,1) \ni x \mapsto \cos(\pi x) \in \bbC$, so $\spb(\Delta^{NL}) \not= -\pi^2$.
	
	We have thus proved that $\spb(\Delta^D) < \spb(\Delta^{NL})$, so the assertion follows from Theorem~\ref{thm:main-result-uniform}.
\end{proof}

\subsection{Comparison of elliptic operators with Neumann boundary conditions} \label{subsec:elliptic-ops-with-neumann-bc}

Let $\Omega \subseteq \bbR^d$ be a bounded domain with Lipschitz boundary or, more generally, with the \emph{extension property} (i.e., every function in $H^1(\Omega)$ can be extended to a function in $H^1(\bbR^d)$). Let $\calA,\hat \calA: \Omega \to \bbR^{d \times d}$ be measurable and essentially bounded mappings. Assume moreover that there exists a constant $\nu > 0$ such that the ellipticity conditions
\begin{align*}
	\langle \xi, \calA(x) \xi \rangle \ge \nu \norm{\xi}^2 \quad \text{and} \quad \langle \xi, \hat \calA(x) \xi \rangle \ge \nu \norm{\xi}^2 \qquad \text{for all } \xi \in \bbR^d
\end{align*}
hold for almost all $x \in \Omega$.

Then we can define elliptic operators $A^N$ and $\hat A^N$ in divergence form that are associated with the coefficient functions $\calA$ and $\hat \calA$, respectively, and have Neumann boundary conditions. More precisely, we let $-A^N$ and $-\hat A^N$ denote the operators on $L^2(\Omega)$ associated with forms
\begin{align*}
	(u,v) \mapsto \int_\Omega  \calA \nabla u \, \nabla \overline{v} \dx x \quad \text{and} \quad (u,v) \mapsto \int_\Omega \hat{\calA} \nabla u \, \nabla \overline{v} \dx x
\end{align*}
on $H^1(\Omega)$, respectively.

Both operators $A^N$ and $\hat A^N$ generate positive, irreducible, contractive and immediately compact semigroups on $L^2(\Omega)$, their spectral bounds equal $0$, and the eigenspaces $\ker A^N$ and $\ker \hat A^N$ are spanned by the constant function $\one$ with value $1$. As $\Omega$ has the extension property, it follows from ultracontractivity theory (see for instance \cite[Sections~7.3.2 and~7.3.6]{Arendt2004}) that $e^{tA^N}$ and $e^{t\hat A^N}$ map $L^2(\Omega)$ into $L^\infty(\Omega) = (L^2(\Omega))_{\one}$ for each $t > 0$.

As a consequence of Theorem~\ref{thm:main-result-individual}, the semigroup $(e^{tA^N})_{t \ge 0}$ does not eventually dominate the semigroup $(e^{t\hat A^N})_{t \ge 0}$ (nor vice versa), unless the operators $A^N$ and $\hat A^N$ coincide:

\begin{theorem}
	If the semigroup $(e^{t \hat A^N})_{t \ge 0}$ individually eventually dominates the semigroup $(e^{t A^N})_{t \ge 0}$, then the operators $\hat A^N$ and $A^N$ coincide.
\end{theorem}
\begin{proof}
	It follows for instance from \cite[Corollary~3.3 and Theorem~5.2]{DanGluKen16b} that both semigroups are individually eventually strongly positive with respect to $\one$. Hence, the assumptions of our Theorem~\ref{thm:main-result-individual} are satisfied. So if $\hat A^N$ and $A^N$ are distinct and if $(e^{t \hat A^N})_{t \ge 0}$ individually eventually dominates $(e^{t A^N})_{t \ge 0}$, then $\spb(\hat A^N) > \spb(A^N)$. But this cannot be true since $\spb(\hat A^N) = 0 = \spb(A^N)$.
\end{proof}

\subsection{Comparison of elliptic operators with Dirichlet boundary conditions} \label{subsec:elliptic-ops-with-dirichlet-bc}

The situation of the previous subsection changes dramatically if we consider Dirichlet instead of Neumann boundary conditions. We consider the same situation as in Subsection~\ref{subsec:elliptic-ops-with-neumann-bc}, with the following three changes:

\begin{itemize}
	\item We assume that $\Omega$ has Lipschitz boundary.
	\item We consider Dirichlet boundary conditions now, i.e., the form domain is no longer $H^1(\Omega)$ but $H^1_0(\Omega)$ instead.
	\item We assume that the matrices $\calA(x)$ and $\hat \calA(x)$ are symmetric for almost every $x \in \Omega$. (This assumption is only needed in order to apply Theorem~\ref{thm:main-result-uniform} instead of Theorem~\ref{thm:main-result-individual}, so that we obtain uniform rather than individual eventual domination.)
\end{itemize}

Denote the corresponding elliptic operators by $A^D$ and $\hat A^D$. Let $0 \le u \in L^2(\Omega)$ and $0 \le \hat u \in L^2(\Omega)$ denote eigenfunctions of $A^D$ and $\hat A^D$ for the eigenvalues $\spb(A^D)$ and $\spb(\hat A^D)$, respectively. Then $u$ and $\hat u$ are strictly positive almost everywhere. We have the following domination result:

\begin{theorem} \label{thm:compare-elliptic-operators-with-dirichlet-bc}
	If $c\hat u \ge u$ for a number $c > 0$, then there exists a number $d > 0$ and a time $t_1 > 0$ such that $e^{t \hat A^D} \ge e^{tdA^D}$ for all $t \ge t_1$.
\end{theorem}
\begin{proof}
	Choose $d > 0$ such that $\spb(\hat A^D) > \spb(d A^D)$. Since $\Omega$ has Lipschitz boundary, it follows from \cite[Lemma~2]{Davies1987} that the operators $e^{t\hat A^D}$ and $e^{td A^D}$ map $L^2(\Omega)$ into the principal ideals $(L^2(\Omega))_{\hat u}$ and $(L^2(\Omega))_{u} \subseteq (L^2(\Omega))_{\hat u}$, respectively. Hence, the semigroup $(e^{t\hat A^D})_{t \ge 0}$ is uniformly eventually strongly positive with respect to $\hat u$ according to \cite[Corollary~3.5]{DanersUNIF} and therefore, we can apply Theorem~\ref{thm:main-result-uniform} to conclude that $e^{t \hat A^D} \ge e^{td A^D}$ for all sufficiently large times $t$.
\end{proof}

Let the function $\delta \in L^2(\Omega)$ be given by
\begin{align*}
	\delta(x) = \dist(x,\partial \Omega)
\end{align*}
for all $x \in \Omega$. If $\Omega$ has sufficiently smoothing boundary and the coefficients $\calA$ and $\hat \calA$ are sufficiently smooth, it follows that the eigenfunctions $u$ and $\hat u$ are both bounded above and below by strictly positive multiples of $\delta$ (see for instance \cite[Section~3]{Davies1987}). Hence, the assumption of Theorem~\ref{thm:compare-elliptic-operators-with-dirichlet-bc} that $c \hat u \ge u$ for some $c > 0$ is satisfied in this situation, and so is the converse inequality $\tilde c u \ge \hat u$ for some $\tilde c > 0$. Thus, in this situation Theorem~\ref{thm:compare-elliptic-operators-with-dirichlet-bc} implies the existence of numbers $b_1,b_2 > 0$ such that
\begin{align*}
	e^{b_1 t \hat A^D} \le e^{t A^D} \le e^{b_2 t \hat A^D}
\end{align*}
for all sufficiently large times $t$.

\subsection{Squares of generators}

This subsection is loosely inspired by \cite[Subsection~6.3]{Daners2016} and \cite[Theorem~6.1]{DanGluKen16b}. We show how our results can be used to derive eventual domination between a semigroup $(e^{tA})_{t \in [0,\infty)}$ and the semigroup generated by the square $A^2$ of $A$. We formulate our result in the setting of self-adjoint operators on Hilbert spaces.

\begin{theorem} \label{thm:square-of-generator}
	Let $(\Omega,\mu)$ be a $\sigma$-finite measure space, set $L^2 := L^2(\Omega,\mu)$ and let $0 \le u \in L^2$ be a function such that $u(\omega) > 0$ for almost all $\omega \in \Omega$. Consider a real and self-adjoint operator $A: L^2 \supseteq D(A) \to L^2$ which satisfies the following two assumptions:
	\begin{itemize}
		\item \emph{Smoothing assumption:} There exists a time $t_0 \in [0,\infty)$ such that $e^{t_0A}L^2 \subseteq (L^2)_u$.
		\item \emph{Spectral assumptions:} The spectral bound $\spb(A)$ satisfies $\spb(A) \le 0$; moreover, $\spb(A)$ is a geometrically simple eigenvalue of $A$ and the eigenspace $\ker(\spb(A) - A)$ is spanned by a vector $v \gg_u 0$.
	\end{itemize}
	Then both semigroups $(e^{tA})_{t \in [0,\infty)}$ and $(e^{-tA^2})_{t \in [0,\infty)}$ are uniformly eventually strongly positive with respect to $u$. Moreover, those semigroups exhibit the following behaviour:
	\begin{enumerate}[\upshape (a)]
		\item If $\spb(A) \in \{-1, 0\}$ and $-A^2 \not= A$, then $(e^{-tA^2})_{t \ge 0}$ does not individually eventually dominate $(e^{tA})_{t \ge 0}$, nor vice versa.
		\item If $\spb(A) \in (-1,0)$, then there exists a time $t_1 \in [0,\infty)$ such that $e^{-tA^2} \ge e^{tA}$ for all  $t \ge t_1$.
		\item If $\spb(A) \in (-\infty,-1)$, then there exists a time $t_1 \in [0,\infty)$ such that $e^{tA} \ge e^{-tA^2}$ for all $t \ge t_1$.
	\end{enumerate}
\end{theorem}
\begin{proof}
	First we note that it follows, for instance, from the spectral theorem that the range of $e^{-t_0A^2}$ is contained in the range of $e^{t_0A}$, so we know that $e^{-t_0A^2} L^2 \subseteq (L^2)_u$. Next, we observe that $\spb(-A^2) = -\spb(A)^2$ and hence,
	\begin{align*}
		\ker\left(\spb(-A^2) - (-A^2)\right) = \ker\left(A^2 - \spb(A)^2\right) = \ker(\spb(A) - A);
	\end{align*}
	for $\spb(A) < 0$ the second equality follows from the fact that $-\spb(A) > 0$ is not an eigenvalue of $A$, and for $\spb(A) = 0$ the second equality follows from the fact that every eigenvalue of a self-adjoint operator is semi-simple (i.e., its algebraic multiplicity coincides with its geometric multiplicity).
	
	We conclude that the spectral bound of $-A^2$ is a geometrically simple eigenvalue of $-A^2$ and that the corresponding eigenspace is spanned by the vector $v \gg_u 0$. We can now employ the characterization result from \cite[Corollary~3.5]{DanersUNIF} which tells us that both semigroups $(e^{tA})_{t \in [0,\infty)}$ and $(e^{-tA^2})_{t \in [0,\infty)}$ are uniformly eventually strongly positive with respect to $u$.
	
	Assertion~(a) now follows from Theorem~\ref{thm:main-result-uniform}. Assertions~(b) and~(c) also follow from this theorem since we have $\spb(-A^2) > \spb(A)$ in~(b) and $\spb(A) > \spb(-A^2)$ in~(c).
\end{proof}

We note in passing that the assumption $-A^2 \not= A$ in assertion~(a) of the above theorem is not redundant: consider for instance the case $L^2 = \bbC^2$ and
\begin{align*}
	A =
	\frac{1}{2}
	\begin{pmatrix}
		-1 & 1 \\ 1 & -1
	\end{pmatrix}.
\end{align*}

Then all assumptions of Theorem~\ref{thm:square-of-generator} are satisfied (for $u = (1,1)^T$) and we have $\spb(A) = 0$ and $-A^2 = A$. 

On the other hand, as the assumptions of Theorem~\ref{thm:square-of-generator} imply that $A$ has compact resolvent, we always have $-A^2 \not= A$ if $L^2$ is infinite-dimensional (and also if $L^2$ is finite-dimensional and the spectrum of $A$ is not a subset of $\{-1,0\}$).

As an example for Theorem~\ref{thm:square-of-generator} we consider the Laplace operator $\Delta^P$ and the bi-Laplace operator on $L^2(0,1)$, both with periodic boundary conditions. The latter operator is precisely given as $-(\Delta^P)^2$ (note that this does not work, in general, for other types of boundary conditions). As a consequence of Theorem~\ref{thm:square-of-generator}(a) (for $u = \one$) we obtain the following result:

\begin{prop}
	The semigroups $(e^{t\Delta^P})_{t \ge 0}$ and $(e^{-t(\Delta^P)^2})_{t\ge 0}$ are uniformly eventually strongly positive with respect to $\one$, but neither of those two semigroup dominates the other one individually eventually.
\end{prop}

We note in passing the well-known fact that actually $e^{t\Delta^P}f \gg_{\one} 0$ for each $0 < f \in L^2(0,1)$ and \emph{each} time $t > 0$, i.e., the semigroup $(e^{t\Delta^P})_{t \ge 0}$ is even \emph{immediately strongly positive} with respect to $\one$. However, it is worthwhile noting that the bi-Laplace semigroup $(e^{-t(\Delta^P)^2})_{t\ge 0}$ is not positive for small times.

\subsection{Diffusion on networks}

We consider a finite connected graph $\mG=(\mV,\mE)$. Upon fixing an arbitrary orientation of $\mG$ each edge $\me\equiv(\mv,\mw)$ can be identified with an interval $[0,\ell_\me]$ and its endpoints $\mv,\mw$ with $0$ and $\ell_\me$, respectively. In such a way one naturally turns the $\mG$ into a metric measure space $\mathcal G$: a \textit{network} whose underlying discrete graph is precisely $\mG$~\cite[Chapter 2]{Mug14}.

We shall be interested in the Laplacian $\Delta^{\mathcal G}$ on $\mathcal G$, which is defined edgewise as the second derivative and whose domain consists in the space of all functions that are in the Sobolev space $H^2(0,\ell_\me)$ on each edge $\me$ and which satisfy so-called \textit{natural conditions} on a $\mathcal{V}$: we impose (i) continuity of the functions at the vertices $v \in \mV$ as well as (ii) the Kirchhoff condition: at each vertex the normal derivatives along all neighbouring edges sum up to 0. 

It is well-known that $\Delta^{\mathcal G}$ is associated with the form
\[
	(u,v)\mapsto \sum_{\me\in\mE}\int_0^{\ell_\me}u'(x)\overline{v'(x)}\ dx,\qquad u,v\in H^1(\mathcal G);
\]
this is a Dirichlet form~\cite{KraMugSik07}, hence the associated heat semigroup is positive. 
(Here $H^1(\mathcal G)$ denotes the space of all functions in $ \bigoplus_{\me\in\mE}H^1(0,\ell_\me)$ which are in addition continuous at the vertices.) Indeed, this semigroup is stochastic on all $\mathcal G$, hence the heat content of any two graphs of same total length is the same at each time.

If $\mathcal G'$ is obtained from $\mathcal G$ by identifying two vertices, then $\Delta_{\mathcal G'}$ and $\Delta_{\mathcal G}$ act on the Hilbert space $L^2(\mathcal G)\simeq L^2(\mathcal G')=\bigoplus_{\me\in \mE}L^2(0,\ell_\me)$, and by~\cite[Prop.~6.70]{Mug14} they have the same spectral bound. Hence, we conclude from Theorem~\ref{thm:main-result-uniform} the following.

\begin{prop}
	Let $\mathcal G$ be a network whose underlying graph is finite and connected. If $\mathcal G'$ is obtained from $\mathcal G$ by identifying two vertices, then neither does $(e^{t\Delta^{\mathcal G}})_{t\in [0,\infty)}$ eventually dominate $(e^{t\Delta^{\mathcal G'}})_{t\in [0,\infty)}$ nor vice versa.
\end{prop}

A few graph operations that do strictly reduce the lowest non-zero eigenvalue of $\Delta^{\mathcal G}$ and hence enforce eventual domination are known: they are more technical and discussing them goes beyond the scope of this note. We refer the interested reader to~\cite[\S~3]{BerKenKur19}.

\subsection{Non-monotonicity of semigroups} \label{subsection:non-monotonicity-of-semigroups}

We conclude the article by getting back to the discussion from Subsection~\ref{subsection:heat-semigroup-at-different-times} (in a more general setting). First we note that a positive semigroup $(e^{tA})_{t \in [0,\infty)}$ on a Banach lattice $E$ cannot dominate the semigroup $(e^{tcA})_{t \in [0,\infty)}$ for any $c \in (1,\infty)$ unless the semigroup operators have a very special structure:

\begin{prop} \label{prop:non-decreasing-orbits}
	Let $E$ be a (real or complex) Banach lattice and let $(e^{tA})_{t \in [0,\infty)}$ be a positive $C_0$-semigroup on $E$. If there exists a number $c \in (1,\infty)$ such that $e^{tcA} \le e^{tA}$ for all times $t \in [0,\infty)$, then $0 \le e^{tA} \le \id_E$ for all $t \in [0,\infty)$.
\end{prop}

The inequality $0 \le e^{tA} \le \id_E$ means that $e^{tA}$ is an element of the so-called \emph{center} of the space of bounded linear operators on $E$. If $E = L^p(\Omega,\mu)$ for $p \in [1,\infty]$ and for a $\sigma$-finite measure space $(\Omega,\mu)$, then this center consists precisely of the multiplication operators (which are given by multiplication with a fixed function from $L^\infty(\Omega,\mu)$). For a few more details we refer for instance to \cite[Section~C-I-9]{Nag86}.

\begin{proof}[Proof of Proposition~\ref{prop:non-decreasing-orbits}]
	By replacing $t$ with $t/c^{n+1}$ we obtain for each $n \in \bbN_0$ the inequality $e^{t/c^n A} \le e^{t/c^{n+1}A}$ for all $t \in [0,\infty)$. Iteration of this inequality yields that $e^{tA} \le e^{t/c^n A}$ for all $n \in \bbN_0$, so the assertion $e^{tA} \le \id_E$ follows by taking the strong operator limit as $n \to \infty$.
\end{proof}

On the other hand Theorem~\ref{thm:main-result-uniform} implies that, in the Hilbert space case and under appropriate technical assumptions, we have $e^{tcA} \ge e^{tA}$ for all sufficiently large times $t$. Let us state this explicitly in the following theorem.

\begin{theorem} \label{thm:eventual-monotonicity}
	Let $L^2 := L^2(\Omega,\mu)$ for a $\sigma$-finite measure space $(\Omega,\mu)$ and let $u \in (L^2)_+$ be a function which is strictly positive almost everywhere. Consider a real and self-adjoint $C_0$-semigroups $(e^{tA})_{t \in [0,\infty)}$  whose generator $A$ satisfies $\spb(A) < 0$. Suppose that $\ker(\spb(A) - A)$ is one-dimensional and contains a vector $v \gg_u 0$, and assume that $e^{t_0A}L^2 \subseteq (L^2)_u$ for at least one time $t_0$.
	
	Then, for every number $c \in (1,\infty)$, there exists a time $t_c \in [0,\infty)$ such that $0 \le e^{tcA} \le e^{tA}$ for all $t \ge t_c$.
\end{theorem}
\begin{proof}
	Fix $c \in (1,\infty)$ and consider the self-adjoint operators $A$ and $cA$. The semigroups $(e^{tA})_{t \in [0,\infty)}$ and $(e^{tA})_{t \in [0,\infty)}$ are uniformly eventually strongly positive with respect to $u$ according to \cite[Corollary~3.5]{DanersUNIF}, and we have $\spb(cA) = c\spb(A) < \spb(A)$. Thus, Theorem~\ref{thm:main-result-uniform} yields the assertion.
\end{proof}

\begin{example} \label{ex:eventual-monotonicity-for-dirichlet-laplacian}
	Let $\Omega \subseteq \bbR^d$ be a bounded domain with Lipschitz boundary. Let $\calA: \Omega \to \bbR^{d \times d}$ be a measurable and essentially bounded mapping. Assume moreover that the matrix $\calA(x)$ is symmetric for almost every $x \in \Omega$ and that there exists a constant $\nu > 0$ such that the ellipticity condition
	\begin{align*}
		\langle \xi, \calA(x) \xi \rangle \ge \nu \norm{\xi}^2 \qquad \text{for all } \xi \in \bbR^d
	\end{align*}
	holds for almost all $x \in \Omega$. Let $A^D$ denote the realization of the operator $v \mapsto \diver(\calA \nabla v)$ on $L^2(\Omega)$ with Dirichlet boundary conditions.
	
	Then $A^D$ is self-adjoint and satisfies $\spb(A^D) < 0$. Its first eigenspace is one-dimensional and spanned by an eigenfunction $u$ that is strictly positive almost everywhere. Since $\Omega$ has Lipschitz boundary, it follows from \cite[Lemma~2]{Davies1987} that the operator $e^{tA^D}$ maps $L^2(\Omega)$ into the principal ideal $(L^2(\Omega))_u$ for each $t > 0$. Thus, it follows from Theorem~\ref{thm:eventual-monotonicity} that, for every number $c \in (1,\infty)$, there exists a time $t_c$ such that $e^{tcA^D} \le e^{tA^D}$ for all $t \ge t_c$. Moreover, we cannot have $t_c = 0$ according to Proposition~\ref{prop:non-decreasing-orbits}.
\end{example}

\subsection*{Acknowledgements}

Around 2015 at Ulm University (Germany), Khalid Akhlil (now in Ouarzazate, Morocco) suggested to the first-named author to seek for characterisations of eventual domination of operator semigroups in the spirit of \cite[Theorem~5.4]{Daners2016} and \cite[Theorem~5.2]{DanGluKen16b}, but they did not further pursue this suggestion. A similar idea came up once again in a discussion between the present authors at the end of 2017, which led to the present article.

We are indebted to Sahiba Arora (Dresden, Germany) for several suggestions which helped us to improve the paper.

\bibliographystyle{plain}
\bibliography{literature}

\end{document}